\documentclass[a4paper,12pt,twoside]{amsart} 
\usepackage{amsmath,amsthm}
\usepackage{amsfonts}
\usepackage{amssymb}
\usepackage[english]{babel}

\newtheorem{theo}{Theorem}[section]
\newtheorem{lem}[theo]{Lemma}
\newtheorem{prop}[theo]{Proposition}
\newtheorem{cor}[theo]{Corollary}

\begin{document}

\baselineskip=15pt

\title
[Spaces with all power-bounded operators almost periodic]
{Reflexive Banach spaces with all \\power-bounded operators almost periodic}

\author{Michael Lin}
\address{Department of Mathematics, Ben-Gurion University, Beer-Sheva, Israel}
\email{lin@math.bgu.ac.il}


\begin{abstract}
We analyze  the ergodic properties of power-bounded operators on a reflexive Banach
space of the form "scalar plus compact-power", and show that they are almost periodic
(all the orbits are conditionally compact).
If such an operator is weakly mixing, then it is stable
(its powers converge in the strong operator topology). 
Let $X_{ISP}$ be the separable reflexive indecomposable Banach space constructed 
by Argyros and Motakis, in which every operator has an invariant subspace.
We conclude that every power-bounded operator on a closed subspace of $X_{ISP}$
is almost periodic.

\end{abstract}
\maketitle

\section{Introduction}

Lorch, Kakutani, and Yosida proved (independently) that if $T$ is a power-bounded 
(linear) operator on a reflexive (real or complex) Banach space $X$, 
then $T$ is {\it mean ergodic}: for every $x \in X$ 
the averages $\frac1n\sum_{k=1}^n T^k x$ converge in norm (as $n \to \infty$); 
$T$ then induces on $X$ the {\it ergodic decomposition}
\begin{equation} \label{decomp}
X= F(T) \oplus \overline{(I-T)X},
\end{equation}
where $F(T):=\{y:\, Ty=y\}$ is the set of fixed points.
The limit $Ex := \lim_{n} \frac1n \sum_{k=1}^n T^k x$ satisfies $ET=E=TE$, and 
is the projection on $F(T)$ corresponding to (\ref{decomp}) -- see \cite[pp. 71-74]{Kr}.
Applying the above to $T^*$, it is easy to check that 
$E(T^*)x^*:= \lim_{n\to \infty} \frac1n \sum_{k=1}^n T^{*k} x^*$ satisfies
$E(T^*)=E^*$.
\smallskip

Fonf, Lin and Wojtaszczyk \cite{FLW} proved that if every power-bounded operator $T$
on a Banach space $X$ with a Schauder basis is mean ergodic,
then $X$ must be reflexive.
In \cite{FLW1} they show examples of non-reflexive Banach spaces with basis such that every
{\it contraction} is mean ergodic.
\smallskip

An operator $T$ on a Banach space is called {\it (weakly) almost periodic} if 
all its orbits $\{T^n x\}_{n\ge 0}$ are conditionally (weakly) compact; a weakly almost
periodic operator is power-bounded.  In fact,
Kakutani and Yosida state and prove that a weakly almost periodic operator is mean ergodic.

Obviously, all power-bounded operators on a reflexive space are weakly almost periodic. 
Conversely, if $X$ is a separable Banach space with basis such that every power-bounded 
operator is weakly almost periodic, then by \cite{FLW} $X$ is reflexive.
\smallskip

A mean ergodic $T$ is called {\it stable} if $\|T^n x -Ex\| \to 0$ for every $x \in X$ 
(i.e. $T^n$ converges to $E$ in the strong operator topology). 
A unitary operator on a Hilbert space is never stable (nor any non-identity isometry 
on a reflexive space); however, it is possible that $T^nx \to Ex$ weakly for every $x \in X$.
For $T$ induced on $L_2$ by an ergodic probability preserving transformation this stability is
called {\it mixing}. Extending the definition in ergodic theory, we will call $T$ on 
a reflexive Banach space {\it weakly mixing} if  for every $x \in X$ we have 
$\lim_n \frac1n \sum_{k=1}^n |\langle x^*,T^k(x-Ex)\rangle| =0$ for every $x^* \in X^*$;
equivalently, using reflexivity, for $x\in X$ there exists a sequence $\{n_i\}$ such that 
$T^{n_i}x \to Ex$ weakly, by \cite{JL1}. Note that we do not require that $F(T)$ be 
one-dimensional. If $T$ is almost periodic and weakly mixing, then it is stable
(since $TE=E$).

The isometry $T$ induced on $L_2[0,1]$ by an ergodic probability preserving transformation 
is "in general" weakly mixing \cite{Ha}, but not mixing \cite{Ro} (hence not almost periodic).
On the other hand, in finite-dimensional spaces weak mixing coincides with stability.
The  (invertible) isometry induced by Chacon's  explicit non-mixing weakly mixing transformation
\cite{Ch} on $L_p[0,1]$, $1<p<\infty$, is a weakly almost periodic operator which is not
almost periodic.
\medskip

The purpose of this note is to show the existence of  separable reflexive (real or complex)
infinite-dimensional Banach spaces on which every power-bounded operator is almost periodic
(so if it is weakly mixing it is in fact stable). The spaces we use are the subspaces of
the space $X_{ISP}$ constructed by Argyros and Motakis \cite{AM}, or certain subspaces of
the space $X_{\mathcal T}$ they constructed in \cite{AM2}; it is implied in the 
construction of $X_{ISP}$ (and explicit in part (vi) of the theorem in \cite[p. 1382]{AM}) 
that the field of scalars is $\mathbb R$, but Dr. Motakis has informed us 
that "with some adjustment of constants" the construction works also for complex scalars.
In the later construction of $X_{\mathcal T}$ by Argyros and Motakis \cite[Theorem A]{AM2}, 
they sepcify explicitly (on p. 629) that their construction "can also be carried out over 
the field of complex numbers".
\medskip

\section{Almost periodicity of "scalar plus compact-power" operators}

In the space $X_{ISP}$ constructed in \cite{AM}, every operator is of the form
$T=\alpha I +S$ with $S^3$ compact. 
In the space $X_{\mathcal T}$ constructed in \cite{AM2}, every operator is of the form
$T=\alpha I +S$ with $S^2$ compact. 
In this section we study operators, in real or complex Banach spaces,
 of the form $\alpha I+S$, with $S^k$ compact for some $k$.

\begin{lem} \label{spectrum}
Let $X$ be an infinite-dimensional complex Banach space and let $T$ be of the form 
$T=\alpha I+S$, where $\alpha \in \mathbb C$ and $S^k$ is compact for some $k \ge 1$. Then 

(i) $\sigma(T) =\alpha +\sigma(S)$.

(ii) $\sigma(T)$ is finite or countable, $\alpha \in \sigma(T)$, $\alpha$ is the only 
accumulation point of $\sigma(T)$, and  every $\lambda \ne \alpha$ in $\sigma(T)$ 
is an eigenvalue of $T$ of finite multiplicity.

(iii) If $T$ is power-bounded, then $|\alpha| \le 1$.
\end{lem}
\begin{proof}
(i) is immediate from $\lambda I-T= (\lambda -\alpha)I -S$.

(ii) Since $S^k$ is compact, the properties follow from those of $S$, given by 
\cite[Theorem VII.4.6, p. 577]{DS}.

(iii) Power-boundedness yields $r(T) \le 1$, and since $\alpha \in \sigma(T)$
by (ii), $|\alpha| \le 1$.
\end{proof}


\begin{prop} \label{stability}
Let $T$ be a power-bounded operator on a complex Banach space $X$ and assume that  
$\sigma(T) \cap \mathbb T$ is countable.
 Then the following are equivalent for $x \in X$:

(i) There exists an increasing subsequence $\{n_j\}$ such that $T^{n_j}x \to 0$ weakly.

(ii) $\frac1n \sum_{k=1}^n |\langle x^*,T^k x \rangle| \to 0$ as $n\to \infty$, for every
$x^* \in X^*$.

(iii) $\|T^n x \| \to 0$.
\end{prop}
\begin{proof}
For any power-bounded $T$,
(i) implies (ii) by the Proposition in \cite{JL1} (which is valid for real or 
complex Banach spaces). Since (iii) obviously implies (i), we have to prove only 
that (ii) implies (iii), under the assumption on $T$.
It is easy to show that 
$$
Z:= \{z \in X:\,  \frac1n\sum_{k=1}^n |\langle x^*,T^k z\rangle | \to 0  \quad
\text{\rm for every }\  x^* \in X^* \}
$$
is a closed subspace of $X$, invariant under $T$,  and $x \in Z$ by (ii).
We show that $\sigma(T_{|Z}) \cap \mathbb T$ is countable. Let $\lambda \in \mathbb T$
with $\lambda I -T$ invertible. Then $\|\frac1n\sum_{k=1}^n (\bar\lambda T)^k\| \to 0$,
which yields that   $\|\frac1n\sum_{k=1}^n (\bar\lambda T_{|Z})^k\| \to 0$; hence
$\lambda I_Z -T_{|Z}$ is invertible on $Z$. Thus $\sigma(T_{|Z}) \cap \mathbb T \subset
\sigma(T) \cap \mathbb T$ is countable.  We can therefore assume $X=Z$.
Hence for $|\lambda|=1$, if $T^*y^* =\lambda y^*$, then $\langle y^*,z\rangle =0$
for every $z \in X$, so $y^*=0$. Thus $T^*$ has no unimodular eigenvalues.

Power-boundedness yields $r(T) \le 1$. If $r(T) <1$, then $\|T^n\| \to 0$ and (iii) holds.

Assume now $r(T)=1$. Since $\sigma(T) \cap \mathbb T$ 
is at most countably infinite and $T^*$ has no unimodular eigenvalues (we have reduced 
the proof to $X=Z$), we can apply Theorem 5.1 of Arendt and Batty \cite{AB} (see also \cite{LV}), 
to conclude that $T^n \to 0$ (on $Z$) in the strong operator topology, so $x$ satisfies (iii).
\end{proof}




\begin{theo} \label{real-stability}
Let $T$ be a power-bounded operator on a real or complex Banach space $X$ and assume that 
$T$ is of the form $T=\alpha I+S$, with $\alpha$ a scalar and $S^k$ compact for some 
$k \ge 1$. Then the following are equivalent for $x \in X$:

(i) There exists an increasing subsequence $\{n_j\}$ such that $T^{n_j}x \to 0$ weakly.

(ii) $\frac1n \sum_{k=1}^n |\langle x^*,T^k x \rangle| \to 0$ as $n\to \infty$, for every
$x^* \in X^*$.

(iii) $\|T^n x \| \to 0$.
\end{theo}
\begin{proof} When $X$ is over $\mathbb C$, the theorem follows from Proposition \ref{stability},
since by Lemma \ref{spectrum}(ii) $T$ has countable spectrum.

Now let $X$ be over $\mathbb R$. As noted in the proof of Proposition \ref{stability}, 
we have to prove (also in the real case), only that (ii) implies (iii).

Let $X_{\mathbb C}=X \oplus X$ be the usual complexification of $X$ \cite[Section 77]{Ha2}, 
and let $T_{\mathbb C}(x,y):=(Tx,Ty)$ be the complexification of $T$.  
On $X_\mathbb C$ we define the  norm  (attributed to Taylor in \cite[Theorem 2]{ MW}; 
see \cite[Proposition 3 and formula (1)]{MST})
\begin{equation}\label{complex-norm}
\|(x,y)\|_T:= \sup _{0\le t \le 2\pi}\|x\cos t -y\sin t\| = 
\sup_{\phi \in X^*, \, \|\phi\| \le 1} \sqrt{\phi(x)^2 +\phi(y)^2}.
\end{equation}
Note that all norms on $X_\mathbb C$ which satisfy $\|(x,0)\|=\|x\|$ and $\|(x,-y)\|=\|(x,y)\|$
are equivalent, and satisfy $\|(x,y)\|_T \le \|(x,y)\| \le 2\|(x,y)\|_T$ \cite[Proposition 3]{MST}.
In the sequel we write $\|(x,y)\|$ for $\|(x,y)\|_T$. Clearly 
$$
\max\{\|x\|,\|y\|\} \le \|(x,y)\| \le \sqrt{\|x\|^2 + \|y\|^2} \le  \|x\| +\|y\|.
$$

By \cite[Proposition 4]{MST} $\|(T_\mathbb C)^n\| = \|(T^n)_\mathbb C\| = \|T^n\|$, so
$T_{\mathbb C}$ is power-bounded.
With the customary abuse of notation, we have 
$$
T_{\mathbb C}(x+iy)=Tx +iTy =\alpha x +Sx+i(\alpha y+Sy)=
(\alpha I_{X_{\mathbb C}}+S_{\mathbb C})(x+iy) \ \ x,y \in X.
$$ 
Since the complexification of $S$ satisfies $(S_{\mathbb C})^k = (S^k)_{\mathbb C}$,
it follows easily from (\ref{complex-norm}) that $(S_{\mathbb C})^k$ is compact.  Hence 
$T_{\mathbb C}$ has countable spectrum, by Lemma \ref{spectrum}(ii).

Let $x \in X$ satisfy (ii). For $\phi \in (X_{\mathbb C})^*$ put 
$x^*(y)=Re\langle \phi,y+i0\rangle$ and $y^*(y)=Im\langle \phi,y+i0\rangle$. Then
$x^*, y^* \in X^*$, and (ii) yields
$$
\frac1n\sum_{k=1}^n |\langle \phi,T_{\mathbb C}^k (x,0)\rangle| \le
\frac1n\sum_{k=1}^n |\langle x^* ,T^k x\rangle| + \frac1n\sum_{k=1}^n |\langle y^* ,T^k x\rangle| 
\underset{n\to\infty}\to 0.
$$
Hence $(x,0) \in X_{\mathbb C}$ satisfies (ii)  with respect to $T_{\mathbb C}$;
Proposition \ref{stability} yields
$\|T^n x\| =\|(T_{\mathbb C})^n(x,0)\| \to 0$, 
\end{proof}

{\bf Remark.} Argyros and Haydon \cite{AH} constructed a real Banach space  $X_K$
on which every operator is of the form $\alpha I +S$ with $S$ compact. Theorem 
\ref{real-stability} applies to any power-bounded operator on $X_K$.

\begin{prop} \label{dual}
Let $T$ be power-bounded on a reflexive (real or complex) Banach space. Then
$T$ is weakly mixing if and only if $T^*$ is weakly mixing.
\end{prop}
\begin{proof} Since $(T^*)^*=T$ by reflexivity, we have
$$
\lim_n \frac1n \sum_{k=1}^n |\langle T^{*k}(I-E(T^*))x^*,x)\rangle| =
\lim_n \frac1n \sum_{k=1}^n |\langle x^*,T^k(x-Ex)\rangle| $$ 
for every $x^* \in X^*$ and $x \in X$, which proves the assertion.
\end{proof}

\begin{theo} \label{stability1}
Let $X$ be a reflexive real or complex Banach space, and let $T$ be a
weakly mixing power-bounded operator. If for some scalar $\alpha$ we have
$T=\alpha I +S$ with $S^k$ compact for some $k$, then $T$ and $T^*$ are stable.
\end{theo}
\begin{proof} Let $E$ be the ergodic projection on $F(T)$ defined in the introduction, so for
$x \in \overline{(I-T)X}$ we have $Ex=0$. We may restrict ourselves to the invariant 
subspace $(I-E)X$, so we assume that $X =  \overline{(I-T)X}$ (and then $E=0$).  
The weak mixing then means that 
$$
\lim_n \frac1n \sum_{k=1}^n |\langle x^*,T^k x\rangle| =0 \quad 
\text{\rm for every }\ x^* \in X^* \ \text{\rm and } \ x\in X.
$$
By Theorem \ref{real-stability},
$\|T^nx\| \to 0$ for $x \in\overline{(I-T)X}$; hence $\|T^nx-Ex\| \to 0$ for $x \in X$.
\smallskip

We have thus obtained that $T$ is stable. Since $S^{*k} =(S^k)^*$ is also compact, 
and $T^*$ is weakly mixing by Proposition \ref{dual}, we can apply the above result to $T^*$
and obtain that $T^*$ is stable.
\end{proof}

{\bf Remarks.} 1. If $T$ is power-bounded on a reflexive complex Banach space $X$, then
$x \in X$ satisfies 
$ \lim_n \frac1n \sum_{k=1}^n |\langle x^*,T^k x\rangle| =0 \quad 
\text{\rm for every }\ x^* \in X^*$ if and only if $\langle y^*,x\rangle =0$ for every
eigenvector $T^*y^*=\lambda y^*$  with $|\lambda|=1$ \cite{JL2}.

2. Since $E(T^*)=E^*$ (when $X$ is reflexive), we have $F(T)\ne \{0\}$ if and only if 
$F(T^*) \ne\{0\}$. When $X$ is over $\mathbb C$, we can apply this to $\lambda T$, 
$|\lambda|=1$, and obtain that $T$ and $T^*$ have the same unimodular eigenvalues.

3. The previous remarks yield that $T$ power-bounded on a reflexive complex Banach space $X$
is weakly mixing if and only if it has no unimodular eigenvalues different from 1; see
\cite[Theorem 9]{JL2}.

\begin{prop} \label{complex-WAP}
Let $T$ be a weakly almost periodic operator on a complex Banach space $X$. If
$\sigma(T) \cap \mathbb T$ is countable, then $T$ is almost periodic.
\end{prop}
\begin{proof}
Since $T$ is weakly almost periodic, by the deLeeuw-Glicksberg theorem 
\cite[Theorem 4.11]{DG} it induces the decomposition
$$
X=\text{clm}\{y \in X: Ty= \lambda y \text{ for some } \lambda \in \mathbb T\}\oplus X_0,
$$
where $X_0 := \{z \in X: 0 \text{ is a weak cluster point of } \{T^n z\}_{n\ge 0}\}$.
Since $T$ is weakly almost periodic,   \cite[Theorem 2]{JL2} yields
$$
X_0=
\{z\in X: \frac1n\sum_{k=1}^n |\langle x^*,T^k z\rangle | \to 0 \ \forall\,  x^* \in X^*\}.
$$
By the assumption on $\sigma(T)$, Proposition \ref{stability} yields the decomposition
\begin{equation} \label{jdg}
X=\text{clm}\{y \in X:\, Ty=
\lambda y\ \text{ for some }\ \lambda \in \mathbb T\}\oplus
\{z \in X:\,  \|T^n z\| \to 0\},
\end{equation}
which implies almost periodicity of $T$.
\end{proof}

\begin{cor}
Let $T$ be a power-bounded operator on a reflexive complex Banach space $X$. 
If $\sigma(T) \cap \mathbb T$ is countable, then $T$ is almost periodic.
\end{cor}

{\bf Remark.} For any power-bounded $T$ on a complex Banach space $X$ (not assumed 
reflexive), the decomposition (\ref{jdg}) is equivalent to almost periodicity of $T$,
by the Jacobs-deLeeuw-Glicksberg decomposition induced by $T$ (\cite[pp. 105-106]{Kr}). 
Jamison \cite[Theorem 3.2]{Ja} proved directly the special case that when $T$ is 
almost periodic, $T$ is stable if and only if it has no unimodular eigenvalues except 1.

\begin{theo} \label{AP}
Let $X$ be a reflexive real or complex Banach space, and let $T$ be a
power-bounded operator. If for some scalar $\alpha$ we have
$T=\alpha I +S$ with $S^k$ compact for some $k$, then $T$ and $T^*$ are almost periodic.
\end{theo}
\begin{proof} We first prove the theorem for $X$ a complex Banach space. 
By reflexivity, $T$ is weakly almost periodic, and by Lemma \ref{spectrum}(ii) 
$\sigma(T)$ is countable, so by Proposition \ref{complex-WAP} $T$ is almost periodic.

Now assume $X$ to be a real Banach space, and let $X_{\mathbb C}$ be its complexification, 
with $T_{\mathbb C}$ the complexification of $T$. We saw in the proof of Theroem 
\ref{real-stability} that $T_{\mathbb C}= \alpha I_{X_{\mathbb C}} +S_{\mathbb C}$
with $(S_{\mathbb C})^k$ compact. By the above result for the complex case,
$T_{\mathbb C}$ is almost periodic, hence so is $T$.

Finally, $T^* =\alpha I +S^*$, so applying the above to $T^*$ yields its almost periodicity. 
\end{proof}

\begin{theo}
There exists an infinite-dimensional separable reflexive (real or complex) Banach space $X$
such that:

(i) Every power-bounded operator on $X$ is almost periodic.

(ii) Every weakly mixing power-bounded operator on $X$ is stable.

(iii) Every power-bounded operator on $X^*$ is almost periodic.

(iv) Every weakly mixing power-bounded operator on $X^*$ is stable.
\end{theo}
\begin{proof}
Let $X=X_{ISP}$ be the space constructed by Argyros and Motakis \cite{AM} (as mentioned in
the introduction, the construction of  \cite{AM} is of a real space, but can be modified to 
yield also a complex space $X_{ISP}^{(\mathbb C)}$), or let $X=X_{\mathcal T}$ \cite{AM2}.
By the corresponding theorems in \cite{AM} and \cite{AM2}, every $T$ on $X$ has the form 
$T= \alpha I +S$ with $S^k$ compact ($k=3$ for $X_{ISP}$, $k=2$ for $X_{\mathcal T}$), 
so Theorems \ref{AP} and \ref{stability1} apply.
\end{proof}



{\bf Remarks.} 1. Every closed subspace of $X_{ISP}$ provides an example, since by \cite{AM}
the operators on subspaces of $X_{ISP}$ have similar properties to those on $X_{ISP}$.
Therefore we can strengthen (i) and (ii) in the theorem to:

{\it (i') Every power-bounded operator on a closed subspace of $X_{ISP}$ is almost periodic.}

{\it (ii') Every weakly mixing power-bounded operator on a closed subspace of $X_{ISP}$ is stable.}

2. Another property of $X_{ISP}$ and $X_{\mathcal T}$, proved in \cite{AM} and \cite{AM2},
respectively, is that every closed subspace is indecomposable (a Banach space $X$ is called 
{\it indecomposable} if in any direct sum decomposition $X=Y\oplus Z$, 
one of the summands is necessarily finite-dimensional). It follows that one of the summands 
in the decomposition (\ref{jdg}), for $X_{ISP}$ or $X_{\mathcal T}$, is finite-dimensional.

3. The theorem shows that from the point of view of ergodic theory, closed
subspaces of $X_{ISP}$ behave similarly to finite-dimensional spaces.

\smallskip

\begin{prop} \label{WAP}
Let $X$ be a real or complex Banach space, and let $T$ be a weakly almost periodic operator. 
If for some scalar $\alpha$ we have $T=\alpha I +S$ with $S^k$ compact for some $k$, 
then $T$ is almost periodic.
\end{prop}
\begin{proof} When $X$ is a complex Banach space, the result follows from Lemma 
\ref{spectrum}(ii) and Proposition \ref{complex-WAP}.
\smallskip

We now prove the proposition when $X$ is a real Banach space. Let $X_{\mathbb C}$
be the complexification of $X$, and $T_{\mathbb C}$ the complexification of $T$, 
acting on $X_{\mathbb C}$, as described in the proof of Theorem \ref{real-stability}.
We show that $T_{\mathbb C}$ is weakly almost periodic on $X_{\mathbb C}$. Let
$\phi \in X_{\mathbb C}^*$ and put $x^*(y) =Re\langle \phi,y+i0\rangle$ and
$y^*(y) =Im\langle \phi,y+i0\rangle$. Then $x^*,y^* \in X^*$, and
$$
\langle \phi,T^n_{\mathbb C}(x+iy)\rangle = 
\langle \phi,T^nx+iT^n y\rangle = x^*(T^nx) +iy^*(T^nx) +ix^*(T^ny) -y^*(T^ny).
$$
By weak almost periodicity and Eberlein's theorem, $\{T^nx\}_n $ and $\{T^ny\}_n$
are weakly sequentially compact, which yields that $\{T^n_{\mathbb C}(x+iy)\}_n$ is
weakly sequentially compact in $X_{\mathbb C}$.
Since $T_{\mathbb C}=\alpha I_{X_{\mathbb C}}+S_{\mathbb C}$ with $(S_{\mathbb C})^k$ compact,
we obtain by the case of complex spaces proved above that $T_{\mathbb C}$ is almost periodic; 
hence $T$ is almost periodic.
\end{proof}

\begin{theo}
There exists a non-reflexive real Banach space $X$ with separable dual such that every
weakly almost periodic operator is almost periodic and every weakly almost periodic operator
on $X^*$ is almost periodic.
\end{theo}
\begin{proof} The space $X_K$ of \cite{AH} has $X_K^* =\ell_1$, and Proposition \ref{WAP}
applies to every weakly almost periodic operator on $X_K$.

Since in $\ell_1$ conditional compactness is equivalent to weak sequential 
compactness \cite[IV.13.3]{DS}, every weakly almost periodic operator on $X_K^*=\ell_1$ 
is almost periodic.
\end{proof}

\begin{theo} \label{invertible}
Let $X$ be a real or complex Banach space, and let $T$ be a weakly almost periodic operator, 
such that for some scalar $\alpha$ we have $T=\alpha I +S$ with $S^k$ compact for some $k$. If 
\begin{equation} \label{norms}
\inf_{n \ge 0} \|T^n x\| >0  \quad \text{for every } x \ne 0
\end{equation}
(in particular if $T$ is an isometry), then $T$ is invertible, and $T^{-1}$ is power-bounded 
and almost periodic.
\end{theo}
\begin{proof} Weak almost periodicity implies that $T$ is power-bounded.

We first prove the theorem for $X$ a complex Banach space. 

Obviously $\||x\||:= \sup_{n \ge 0} \|T^n x\|$ is an equivalent norm, with $\||T\|| \le 1$.
If $\||T\|| <1$, then $\||T^n\|| \to 0$, which contradicts (\ref{norms}).
Hence $\||T\||=1$. We may therefore assume for the proof that $\|T\|=1$.

By Proposition \ref{WAP} $T$ is almost periodic, and therefore its Jacobs-deLeeuw-Glicksberg
decomposition is (\ref{jdg}), but in view of (\ref{norms}) we have 
\begin{equation} \label{jdg1}
X=\text{clm}\{y \in X:\, Ty= \lambda y\ \text{ for some }\ \lambda \in \mathbb T\}.
\end{equation}
By \cite[Proposition 8(c)]{JL2}, $T$ is an invertible isometry, and $T^{-1}$ is easily seen 
to be almost periodic. 

Returning to the original norm, $T^{-1}$ is power-bounded.
\smallskip

We now prove the theorem for $X$ real; let $X_{\mathbb C}$ be its complexification,
with $T_{\mathbb C}$ the complexification of $T$. It is easily seen that also 
$T_{\mathbb C}$ satisfies the hypotheses of the theorem, so by the previous part 
$T_{\mathbb C}$ is invertible, with $T_\mathbb C^{-1}$ power-bounded. 
Let $x \in X$ and put $T^{-1}_{\mathbb C}(x,0)=(y,z)$. Then
$$
(x,0) =T_{\mathbb C}T_{\mathbb C}^{-1}(x,0) = T_{\mathbb C}(y,z) =(Ty,Tz),
$$ 
which yields $Ty=x$ (and $Tz=0$, so by (\ref{norms}) $z=0$). Hence $T$ is onto $X$, and one-to-one 
by (\ref{norms}), so it is invertible. Since $T^{-1}x = T_{\mathbb C}^{-1} (x,0)$, we have that 
$T^{-1}$ is power-bounded  and almost periodic, since $T_{\mathbb C}$ is.
\end{proof}

\begin{cor} \label{similar}
 Let $T$ be power-bounded on a (real or complex) Hilbert space $H$,
such that for some scalar $\alpha$ we have $T=\alpha I +S$ with $S^k$ compact for some $k$.
If $\inf_n \|T^n x\| >0$ for every $x \ne 0$, then $T$ is similar to a unitary operator.
\end{cor}
\begin{proof} Since by the theorem $\sup_{n \in \mathbb Z} \|T^n\| < \infty$, the assertion
follows from \cite{SN}.
\end{proof}


\begin{cor} \label{similar1}
Let $X$ be $X_{\mathcal T}$ or a closed subspace of $X_{ISP}$. 
Then every isometry on $X$ is invertible.
\end{cor}

{\bf Remark.} By Jarosz \cite{Jar}, any  Banach space has an equivalent norm 
in which the only invertible isometries are $\lambda I$. 
In that norm $\alpha I +S$ is a weakly almost periodic isometry if and only if $S=0$.ZZ
\smallskip

{\bf Problem.} Let $T$ be a power-bounded operator as in Theorem \ref{invertible}
(so $T$ is invertible with $T^{-1}$ power-bounded). Is $T$ polynomially bounded?
\medskip

{\bf Remarks.} 1. K\'erchy and van Neerven \cite{KV} proved that if $T$ on a complex 
Banach space $X$ is a polynomially bounded operator which satisfies (\ref{norms}), and 
$\sigma(T)\cap \mathbb T$ has Lebesgue measure zero, then there exists an isomorphism $J$
of $X$ onto a Banach space $Y$ and an invertible isometry $V$ on $Y$ such that
 $T=J^{-1}VJ$; hence $T$ is invertible and $T^{-1}$ is power-bounded.
 In Theorem \ref{invertible} we replace polynomial boundedness by 
the representation $\alpha I +S$ (which yields the spectral condition). 

2. Zarrabi \cite[Theorem 3.1]{Z1} proved that if an invertible contraction  (power-bounded)
$T$ on a complex Banach space has countable spectrum  $\sigma(T) \subset \mathbb T$ and 
$\frac{\log \|T^{-n}\|}{\sqrt{n}} \to 0$, then $T$ is an isometry ($T^{-1}$ is power-bounded).
If, in addition, the countable spectrum is a Helson subset of $\mathbb T$, then $T$ is 
polynomially bounded \cite[Theorem 4.2]{Z2}.

3. Beauzamy and Casazza \cite{BC} (see \cite[pp. 28-31]{CT}) studied the structure of 
isometries on the (real) Tsirelson space, and proved that any isometry on that space is 
invertible. On $X_{ISP}$ or $X_{\mathcal T}$, in any equivalent norm every isometry is invertible.
\medskip

{\bf Acknowledgement.} I am grateful to  Jochen Gl\"uck for pointing out an inaccuracy in the
definition of the complexification norm in the first draft of the paper, and for supplying reference 
\cite{MST}. Thanks also to Guy Cohen  for his comments, which improved the presentation in the paper, 
and to Bill Johnson for the remark to Corollary \ref{similar1}. I am indebted to Spyros Argyros
for pointing out \cite{AM2}.

\end{document}